\documentclass[11pt]{article}

\usepackage{mathtools}
\usepackage{amsthm}
\usepackage{amsmath}
\usepackage{amssymb}
\usepackage{xcolor}
\usepackage{xypic}
\usepackage{cite}
\usepackage{float}
\usepackage{filecontents}
\usepackage{graphicx}
\usepackage{epstopdf}
\usepackage{epsfig}

\definecolor{darkgreen}{rgb}{0,.7,.3}

\theoremstyle{definition}
\newtheorem{definition}{Definition}[section]
\newtheorem{remark}[definition]{Remark}
\newtheorem{algorithm}[definition]{Algorithm}
\theoremstyle{plain}
\newtheorem{lemma}[definition]{Lemma}
\newtheorem{theorem}[definition] {Theorem}
\newtheorem{proposition}[definition] {Proposition}
\newtheorem{main}[definition]{Main Theorem}
\newtheorem{corollary}[definition]{Corollary}

\begin{document}
\title{A Magnus theorem for some amalgamated products}
\author{Carsten Feldkamp}
\date{\today}
\maketitle
\abstract{A group $G$ possesses the Magnus property if for every two elements $u,v \in G$ with the same normal closure, $u$ is conjugate in $G$ to $v$ or $v^{-1}$.
We prove the Magnus property for some amalgamated products including the fundamental group of a closed non-orientable surface of genus 3.
This answers a question of O. Bogopolski and K. Sviridov, who obtained the analogous result for genus $g > 3$.}

\section{Introduction}

A group $G$ possesses the Magnus property if for every two elements $u,v \in G$ with the same normal closure, $u$ is conjugate in $G$ to $v$ or $v^{-1}$.
The Magnus property was named after W. Magnus who proved the so-called Freiheitssatz (see Theorem \ref{Magnus}) and the Magnus property for free groups \cite{ArtMagnus}. Since then, many mathematicians proved or disproved the Freiheitssatz and the Magnus property for certain classes of groups (see e.g. \cite{ArtSurGr}, \cite{ArtOneRel}, \cite{ArtEdjvet}, \cite{ArtHowie81}, \cite{ArtHowSurGr}).

Let $\pi_{1}(S_{g}^{+})$, respectively $\pi_{1}(S_{g}^{-})$, be the fundamental group of the compact orientable, respectively non-orientable, surface of genus~$g$.
The Magnus property of $\pi_{1}(S_{g}^{+})$ for all $g$ was proved independently by O. Bogopolski \cite{ArtSurGr} (by using algebraic methods) and by J.~Howie \cite{ArtHowSurGr}
(by using topological methods). As observed in \cite{ArtSurGr}, there is a third, model theoretic method: two groups $G_1$, $G_2$
are called elementarily equivalent if their elementary theories coincide: $Elem(G_1)=Elem(G_2)$, see \cite{ArtTheorie}.
It is easy to show that elementarily equivalent groups either both possess
the Magnus property, or both do not possess it.
Since the groups $\pi_1(S_g^+)$ for $g\geqslant 2$ and $\pi_1(S_g^-)$ for $g\geqslant 4$ are elementarily equivalent
to the free group on two generators, all these groups possess the Magnus property.

In \cite{ArtOneRel}, O. Bogopolski and K. Sviridov proved the following theorem:

\begin{theorem} \label{ArtikelBSMain} \emph{\textbf{\cite[Main Theorem]{ArtOneRel}}}
Let $G=\langle a,b,y_{1}, \dots , y_{e} \mid [a,b]uv \rangle$, where $e \geqslant 2$, $u,v$ are non-trivial reduced words in letters $y_{1}, \dots, y_{e}$, and $u,v$ have no common letters. Let $r,s \in G$ be two elements with the same normal closures. Then $r$ is conjugate to $s$ or $s^{-1}$.
\end{theorem}

As a corollary of that theorem (\!\!\cite[Corollary 1.3]{ArtOneRel}), they showed the Magnus property of $\pi_{1}(S_{g}^{-})$ for  $g \geqslant 4$. Since the Magnus property trivially holds for genus $1$ and $2$, the authors asked, whether it also holds for the fundamental group of the non-orientable surface of genus 3. With our Main Theorem, that proves the Magnus property for a slightly larger subclass of one-relator groups than in \mbox{Theorem \ref{ArtikelBSMain}}, we answer this question positively. The difficulty with genus $g=3$
is essential since it is well known that the group $\pi_1(S_3^-)=\langle x,y,z\,|\, x^2y^2z^2\rangle$ is not even existentially equivalent to a free group $F_{n}$ on $n$ generators. However, in large parts, our proof follows the proof of the Main Theorem in \cite{ArtOneRel}.

\begin{main} \label{Main}
Let $G=\langle a,b,y_{1},\dots,y_{n} \mid [a,b]u \rangle$, where $n\in \mathbb{N}$ and $u$ is a non-trivial reduced word in the letters $y_{1},\dots,y_{n}$. Then $G$ possesses the Magnus property.
\end{main}

Using the isomorphism $\varphi$ between $\pi_1(S_3^-) = \langle x,y,z \, | \, x^2y^2z^2 \rangle$ and $\widetilde{H}=\langle a,b,c \mid [a,b]c^{2} \rangle$ defined by $\varphi(x)=ca^{-1}, \varphi(y)=b^{-1}c^{-1}$ and $\varphi(z)=cbcac^{-1}$, we get the following corollary.

\begin{corollary} \label{Corgen3}
The group $\pi_{1}(S_{3}^{-})$ possesses the Magnus property.
\end{corollary}

In the proof of Theorem \ref{ArtikelBSMain}, the authors of \cite{ArtOneRel} used an automorphism of $G$ with certain convenient properties. This automorphism is in general absent for the group given in our Main Theorem. For example, it is absent for the group $\pi_1(S_3^-)$. So we introduce an additional tool which we call $\alpha$- and $\omega$-limits (see Section \ref{secalphaomega}).
Together with the results of \cite{ArtOneRel,ArtSurGr,ArtHowSurGr} Corollary \ref{Corgen3} implies the following.

\begin{corollary} The fundamental groups of all compact surfaces possess the Magnus property.\end{corollary}

We will start the proof of our main theorem by recapitulating the notation of \cite[Section 3]{ArtOneRel} with some small alterations. In Section 3 we introduce $\alpha$- and $\omega$-limits, $\alpha$-$\omega$-length $|| \cdot ||_{\alpha,\omega}$ and suitable elements. In Section 4 we give the proof for the case, where the $\alpha$-$\omega$-length of suitable elements is positive (the consideration of this case is close to that in \cite{ArtOneRel}). In Section 5 we complete the proof in the remaining case.

\section{Reduction to a new group and left/right bases} \label{leftright}

We denote the normal closure of an element $g$ in a group $G$ by $\langle \! \langle g \rangle \! \rangle_{G}$ and the exponent sum of an element $g \in G$ in a letter $x$ by $g_{x}$. Note that this sum is well-defined if all relations of $G$ considered as words of a free group have exponent sum $0$ in $x$. Our main theorem can be deduced from the following proposition in the same way as in the Main Theorem of \cite{ArtOneRel} from \cite[Proposition 2.1]{ArtOneRel}. Therefore we leave this argumentation out.

\begin{proposition} \label{prop1} Let $H=\langle x,b,y_{1},\dots,y_{n} \ | \ [x^{k},b]u \rangle$, where $k,n \ge 1$ and $u$ is a non-trivial reduced word in the letters $y_{1},\dots,y_{n}$. Further, let $r,s \in H \backslash \{1\}$ with $r_{x}=0$. Then $\langle \! \langle r \rangle \! \rangle_{H} = \langle \! \langle s \rangle \! \rangle_{H}$ implies that $r$ is conjugate to $s$ or $s^{-1}$.
\end{proposition}

We briefly summarise the concept of left and right bases from \cite{ArtOneRel}. For a given group $G$ and an element $g \in G$ we denote by $g_{i}$ the element $x^{-i}gx^{i}$ for $i \in \mathbb{Z}$.

Let $H$ be as in Proposition \ref{prop1}. We consider the homomorphism $\varphi: H \rightarrow \mathbb{Z}$ which sends $x$ to $1$ and $b,y_{1},y_{2},\dots, y_{n}$ to $0$. For each $i \in \mathbb{Z}$ let $Y_{i}=\{ b_{i}, y_{1,i},y_{2,i},\dots, y_{n,i} \}$, where $y_{j,i}:=x^{-i}y_{j}x^{i}$. Using the rewriting process of Reidemeister-Schreier, we obtain that the kernel $N$ of $\varphi$ has the presentation
\begin{eqnarray*}
N = \langle \, \underset{i \in \mathbb{Z}}{\bigcup} Y_{i} \ | \ b_{i}u_{i}=b_{i+k} \ (i \in \mathbb{Z}) \, \rangle.
\end{eqnarray*}

The group $N$ is the free product of the free groups $G_{i}=\langle Y_{i} \ | \ \rangle \ (i \in \mathbb{Z})$ with amalgamation, where $G_{i}$ and $G_{i+k}$ are amalgamated over a cyclic group $Z_{i+k}$ that is generated by $b_{i}u_{i}$ in $G_{i}$ and by $b_{i+k}$ in $G_{i+k}$. This gives $N=N_{t} \ast \dots \ast N_{t+k-1}$ for all $t \in \mathbb{Z}$, where \[ N_{l}=\dots \ \underset{ Z_{l-k}}{\ast} \ G_{l-k} \ \underset{Z_{l}}{\ast} \ G_{l} \underset{Z_{l+k}}{\ast} G_{l+k} \ \underset{Z_{l+2k}}{\ast} \dots \] for $l \in \mathbb{Z}$.

\begin{proposition} \label{Nfree}
For every $i \in \mathbb{Z}$, the group $N$ is free with basis
$$\mathcal{B}(i) =\{b_{i},b_{i+1}, \dots,b_{i+k-1}\} \, \cup \{y_{m,j} \mid 1 \leqslant m \leqslant n, j \in \mathbb{Z}\}.$$
\end{proposition}

\begin{proof} Since $N=N_{i} \ast \dots \ast N_{i+k-1}$, it suffices to show that each $N_l$ has the basis
$$B_l:=\{b_l\}\cup \{y_{m,l+tk}\,|\, 1 \leqslant m \leqslant n, t\in \mathbb{Z}\}.$$
We have $N_{l}= \underset{p \in \mathbb{N}}{\bigcup} N_{l,p}$, where $N_{l,p}= G_{l-pk} \underset{Z_{l-(p-1)k}}{\ast} \dots \ \underset{Z_{l}}{\ast} G_{l}  \underset{Z_{l+k}}{\ast}   \dots  \underset{Z_{l+pk}}{\ast} G_{l+pk}$. Using Tietze transformations, one can show that $N_{l,p}$ is free with basis $B_{l,p}:=\{b_{l}\} \cup \{y_{m,l+tk} \mid 1 \leqslant m \leqslant n, -p \leqslant t \leqslant p \}$. Since $B_{l,p} \subset B_{l,p+1}$, the group $N_{l}$ is free with basis $B_l= \underset{p \in \mathbb{N}}{\bigcup} B_{l,p}$.
\end{proof}

\noindent
{\bf Notation.} Let $G_{i,j} = \langle G_{i},G_{i+1} \dots,G_{j} \rangle$, $G_{i,\infty}= \langle G_{l} \mid l \geqslant i \rangle$ and $G_{-\infty,i}=\langle G_{l} \mid l \leqslant i \rangle$ for all $i,j \in \mathbb{Z}$ with $i \leqslant j$.
Further, we use the
$$\begin{array}{llllllll}
\emph{b-left basis} & \mathcal{B}^{+}(i) & := & \{b_{i},b_{i+1},\dots,b_{i+k-1} \} & \cup & \{y_{m,j} \mid 1 \leqslant m \leqslant n,j \geqslant i \} & \text{of} & G_{i,\infty} \ \ \text{and the} \\
\emph{b-right basis} & \mathcal{B}^{-}(i) & := &  \{b_{i-k+1},b_{i-k+2},\dots,b_{i} \}  & \cup & \{y_{m,j} \mid 1 \leqslant m \leqslant n,j \leqslant i \} & \text{of} & G_{-\infty,i}. \\
\end{array}$$

\begin{remark} \label{remchange}
Let $r \in N$ be written as a word in the letters $b_{j},y_{m,j} \ (1 \leqslant m \leqslant n,j \in \mathbb{Z})$. For an arbitrary $i \in \mathbb{Z}$ we describe how to rewrite $r$ in the basis $\mathcal{B}^{+}(i)$. First, replace each letter $b_{j}$ of $r$ by $b_{j-k}u_{j-k}$, if $j > i+k-1$, and by $b_{j+k}u_{j}^{-1}$, if $j< i$, and reduce. If the resulting word contains $b_{j}$ that do not belong to $\mathcal{B}^{+}(i)$, we repeat this procedure. After finitely many steps, we will obtain the desired form of~$r$. Analogously, we can rewrite $r$ in the basis $\mathcal{B}^{-}(i)$.
\end{remark}

\section{\boldmath{$\alpha$}- and \boldmath{$\omega$}-limits} \label{secalphaomega}

We keep using the notation introduced in Section $2$. An arbitrary element $r\in N\setminus \{1\}$
can be written in many ways as a reduced word in letters $b_i,y_{m,i}$. For example, $b_0u_0u_1^{-1}b_1^{-1}=b_kb_{k+1}^{-1}=
b_{2k}u_k^{-1}u_{k+1}b_{2k+1}^{-1}$.
We give an algorithm which finds a word $r^{*}$ representing $r$
such that the smallest index of letters used in $r^{*}$ is maximum possible. We call such index the \fbox{{\it $\alpha$-limit} of $r$} and denote it by $\alpha_{r}$. In other words, $\alpha_{r}$ is the largest index such that $r$ is an
element of $G_{\alpha_{r}, \infty}$. The following algorithm rewrites an arbitrary word $r \in N\setminus \{1\}$ into the presentation $r^{*}$ of $r$ written in basis $\mathcal{B}^{+}(\alpha_{r})$ (see Lemma \ref{Lemefficiency}).
For a word $r$ in the alphabet $\{b_i\,|\, i\in \mathbb{Z}\}\cup \{y_{m,i}\,|\, 1 \leqslant m \leqslant n,i\in \mathbb{Z}\}$, let $\min(r)$
denote the minimal index of letters of $r$.

\begin{algorithm} \label{algorithm}
Let $r \in N \backslash \{1\}$. Suppose that $r$ is given as a finite word in letters $b_i,b_{i+1},\dots $ and $y_{m,i},y_{m,i+1},\dots$ ($1 \leqslant m \leqslant n$). In particular, $r\in G_{i,\infty}$.
Let $r[0]$ be the reduced word representing $r$ in basis $\mathcal{B}^{+}(i)$. Increasing~$i$ if necessary, we may assume that $i=\min(r[0])$.

\begin{itemize}
\item[(1)] Let $r[1]$ be the word obtained from $r[0]$ by replacement of each occurrence of the letter $b_i$ by $b_{i+k}u_{i}^{-1}$ followed by free reduction.
Then $r[1]$ presents $r$ in the following basis of $G_{i,\infty}$:
\begin{align} \label{Basisplus1}
\{b_{i+1}, b_{i+2}, \dots, b_{i+k} \} \cup \{y_{m,l} \mid 1 \leqslant m \leqslant n,i \leqslant l \}.\tag{$\star$}
\end{align}

\begin{enumerate}
\item[(1a)] If $r[1]$ does not contain a letter $y_{m,i}$ ($1 \leqslant m \leqslant n$), we reset $r[0]:=r[1]$, $i:=\min(r[1])$, and go back to $(1)$. Clearly, the new $i$ is larger than the old one.
\item[(1b)] If $r[1]$ contains a letter $y_{m,i}$ ($1 \leqslant m \leqslant n$),  the algorithm ends with $r^{*}=r[0]$ and $\alpha_r=i$.
\end{enumerate}
\end{itemize}
\end{algorithm}

\begin{lemma} \label{Lemefficiency}
The output $r^{*}$ of Algorithm \ref{algorithm} coincides with the word representing $r$ in Basis $\mathcal{B}^{+}(\alpha_{r})$, where $\alpha_{r}$ is the $\alpha$-limit of $r$.
\end{lemma}

\begin{proof}
We shall prove that the algorithm ends and that the integer assigned to $\alpha_{r}$ by the algorithm is really the $\alpha$-limit of $r$.

The algorithm ends since the length of $r[0]$ decreases with each iteration of (1) which does not terminate the algorithm.
This can be verified in the following way:
If we arrive in (1a), then the word $r[1]$ does not contain a letter $y_{m,i}$ ($1 \leqslant m \leqslant n$).
Recall that $r[1]$ was obtained from $r[0]$ by replacing all $b_{i}$ with $b_{i+k}u_{i}^{-1}$ and reducing the resulting word. The fact that there is no letter $y_{m,i}$ ($1 \leqslant m \leqslant n$) left in $r[1]$
means that each $b_{i}$ in $r[0]$ occurred in a subword of the form $b_{i}u_{i}$.
Therefore $r[1]$ is shorter than $r[0]$.

Finally, we show that $r$ is not an element of $G_{i+1,\infty}$, where $i$ as in (1b).
Observe that $$G_{i,\infty}=G_{i+1,\infty}\ast \langle y_{m,i} \mid 1 \leqslant m \leqslant n \rangle.$$ By (1b), $r$, written in basis \eqref{Basisplus1},
uses a letter $y_{m,i}$ ($1 \leqslant m \leqslant n$). Therefore $r\notin G_{i+1,\infty}$ and $\alpha_r=i$.
\end{proof}

\begin{corollary} \label{Corfinite}
Let $r$ be an element from $N \backslash \{1\}$ given as a word in some $b$-left basis. Starting \mbox{Algorithm \ref{algorithm}} with $r$, we either get the same presentation or a presentation of shorter length.
\end{corollary}

\begin{proof}
In the proof of Lemma~\ref{Lemefficiency}, we already showed that the length of $r[0]$ in Algorithm \ref{algorithm} decreases with each iteration of (1) which does not terminate the algorithm. Moreover, the last iteration of \mbox{Algorithm~\ref{algorithm}} does not change the current presentation $r[0]$.
\end{proof}

Analogously, we can find a presentation $r_{*}$ of a word $r \in N \backslash \{1\}$ such that the largest used index of letters in $r_{*}$ is minimum possible. We call this index \fbox{{\it $\omega$-limit} of $r$} and denote it by $\omega_{r}$. In
other words, $\omega_{r}$ is the smallest index such that $r$ is an element of $G_{-\infty, \omega_{r}}$. The algorithm to find $r_{*}$ can be received from Algorithm~\ref{algorithm} by ``mirroring'' this algorithm using the $b$-right basis of $G_{-\infty,i}$  and replacements of $b_i$ by $b_{i-k}u_{i-k}$. We define the \fbox{\emph{$\alpha$-$\omega$-length} of $r$} by $||r||_{\alpha,\omega}:=\omega_{r}-\alpha_{r}+1$. Note that the $\alpha$-$\omega$-length of a non-trivial element can be non-positive.

\noindent
\textbf{Examples.}
\begin{itemize}
\item[(i)] Let $r=b_{-2}u_{-2}y_{1,0}b_{4}u_{1}^{-1}$ and $k=3$. To determine $\alpha_{r}$, we write $r=b_{-2}u_{-2}y_{1,0}b_{4}u_{1}^{-1} = b_{1}y_{1,0}b_{4}u_{1}^{-1}$
and get $\alpha_{r}=0$. For $\omega_{r}$ we have $r=b_{-2}u_{-2}y_{1,0}b_{4}u_{1}^{-1} =b_{-2}u_{-2}y_{1,0}b_{1}u_{1}u_{1}^{-1}=b_{-2}u_{-2}y_{1,0}b_{1}$ \mbox{$=b_{-2}u_{-2}y_{1,0}b_{-2}u_{-2}.$}
Thus, $\omega_{r}=0$ and $||r||_{\alpha,\omega}=1$. Note that $r \notin G_{\alpha_{r},\omega_{r}}=G_{0}$.
\item[(ii)] Let $r=b_{5}b_{6}^{-1}$ and $k=4$. Clearly, $\alpha_{r}=5$. To determine $\omega_{r}$, we write $r=b_{5}b_{6}^{-1}= b_{5}u_{2}^{-1}b_{2}^{-1}$ $= b_{1}u_{1}u_{2}^{-1}b_{2}^{-1}=b_{1}u_{1}u_{2}^{-1}u_{-2}^{-1}b_{-2}^{-1}$. Thus, $\omega_{r}=2$ and $||r||_{\alpha,\omega}=-2$.
\end{itemize}

\begin{lemma} \label{LemEig}
Let $r \in N$.
Then $\alpha_{r_{i+j}}=\alpha_{r_{i}}+j$ and $\omega_{r_{i+j}}=\omega_{r_{i}}+j$ for all $i,j \in \mathbb{Z}$. In particular, $||r_{i}||_{\alpha,\omega}=||r_{j}||_{\alpha,\omega}$ for all $i,j \in \mathbb{Z}$.
\end{lemma}

\begin{proof}
Indeed, $r_{i+j}$ can be obtained from $r_i$ by increasing the (second) indices of all letters by $j$.
\end{proof}

\noindent
\textbf{Notation.} Let $r \in N$. For the following two lemmata, let $r(i)$
be the presentation of $r$ written in basis $\mathcal{B}(i)$.

\begin{lemma} \label{Lemcyclical}
Let $r \in N \backslash \{1\}$. The following statements are equivalent:
\begin{itemize}
\item[{\rm (1)}] For some $i\in \mathbb{Z}$, the word $r(i)$ begins with a positive power of a $b$-letter.
\item[{\rm (2)}] For all $i\in \mathbb{Z}$, the word $r(i)$ begins with a positive power of a $b$-letter.
\end{itemize}
\end{lemma}

\begin{proof} The word $r(i+1)$ can be obtained from $r(i)$, by replacement of each occurrence of $b_i$ in $r(i)$ by $b_{i+k}u_i^{-1}$
followed by free reduction. The new letter $b_{i+k}$ does not lie in $\{b_i,\dots,b_{i+k-1}\}$. That prevents cancellation between $b$-letters in $r(i+1)$. 
Therefore $r(i)$ starts with a positive exponent of a $b$-letter if and only if $r(i+1)$ starts with a positive exponent of a $b$-letter. This proves the equivalence $(1)\Leftrightarrow (2)$.
\end{proof}

\begin{corollary} \label{lemdef}
For every $r \in N \backslash \{1\}$ there exists a conjugate $\widetilde{r}$ of $r$ 
such that $\widetilde{r}(i)$ is cyclically reduced for each $i\in \mathbb{Z}$.
\end{corollary}

\begin{proof}
Using conjugation, we may assume that $r(0)$ is cyclically reduced.
If $r(0)$ contains only $y$-letters, we are done with the element $\widetilde{r}$ represented by $r(0)$.
Suppose that $r(0)$ contains a $b$-letter. Let $\widetilde{r}$ be the element represented by a cyclic permutation of $r(0)$ which either starts with a positive power of a $b$-letter, or ends with a negative power of a $b$-letter (but not both).
By Lemma~\ref{Lemcyclical}, this $\widetilde{r}$ has the desired property.
\end{proof}

\begin{definition} \label{defsc}
Let $r \in N\backslash \{1\}$. Any element $\widetilde{r}$ as in Lemma~\ref{lemdef} is called \emph{suitable conjugate for $r$}.
\end{definition}

\begin{remark} {\bf (dual structure of $N$)} \label{Remdual}
Denote $b'_i:=b_{-i}u_{-i}$, $y_{m,i}':=y_{m,-i}^{-1}$ ($1 \leqslant m \leqslant n$) and $G_i':=G_{-i}$.
We call the elements $b_i',y_{m,i}'$ ($1 \leqslant m \leqslant n$) {\it dual} to $b_i,y_{m,i}$ ($1 \leqslant m \leqslant n$) and the subgroup $G_i'$ dual to $G_i$.
Expressing $b$-letters and $y$-letters via their dual, we obtain $b_i=b_{-i}'u_{-i}'$, $y_{m,i}=y_{m,-i}'^{-1}$ ($1 \leqslant m \leqslant n$), where $u'_{-j}$ ($j \in \mathbb{Z}$) is the word obtained from $u_{j}$ by replacing each letter $y_{m,j}$ with $y_{m,-j}'^{-1}$. That justifies the terminology. 

Observe that the old relations $b_iu_i=b_{i+k}$ $(i\in \mathbb{Z})$ can be rewritten in dual letters as $b_i'u_i'=b_{i+k}'$ $(i\in \mathbb{Z})$. Thus, the relations preserve their form. 
Moreover, we have $G_i'=\langle b_i',y_{m,i}' \mid 1 \leqslant m \leqslant n \rangle$ that repeats the form $G_i=\langle b_i,y_{m,i} \mid 1 \leqslant m \leqslant n \rangle$.
Other dual objects can be defined analogously: 
For example, $G_{i,\infty}':=\langle G_i',G_{i+1}'\dots \rangle=G_{-\infty,-i}$ 
and $G_{-\infty,i}':=\langle \dots ,G_{i-1}',G_i'\rangle =G_{-i,\infty}$. We use the following bases of $N$\,
($i\in \mathbb{Z}$):
$$\mathcal{B}(i)' =\{b'_{i},b'_{i+1}, \dots,b'_{i+k-1}\} \, \cup \{y'_{m,j} \mid 1 \leqslant m \leqslant n, j \in \mathbb{Z}\}.$$
along with the
$$\begin{array}{llllllll}
\emph{$b'$-left basis} & \mathcal{B}^{+}(i)' & := & \{b_{i}',b_{i+1}',\dots,b_{i+k-1}' \} & \cup & \{y_{m,j}' \mid 1 \leqslant m \leqslant n, j \geqslant i \} & \text{of} & G_{i,\infty}'\ \text{and the} \\
\emph{$b'$-right basis} & \mathcal{B}^{-}(i)' & := &  \{b_{i-k+1}',b_{i-k+2}',\dots,b_{i}' \}  & \cup & \{y_{m,j}' \mid 1 \leqslant m \leqslant n, j \leqslant i \} & \text{of} & G_{-\infty,i}'. \\
\end{array}$$

For $r\in N\backslash \,\{1\}$, let $\alpha'_r$ be the largest index $i$ such that $r$ is an element of $G'_{i,\infty}$ and let  $\omega'_r$ be the smallest index $i$ such that $r$ is an element of $G'_{-\infty,i}$. We denote $||r||_{\alpha',\omega'}:=\omega_r'-\alpha_r'+1$.
\end{remark}

\begin{lemma} \label{LemDual} For any non-trivial $r\in N$ the following statements are valid.  
\begin{enumerate}
\item[\rm (1)] We have $\alpha_r'=-\omega_{r}$, $\omega_r'=-\alpha_{r}$, and $||r||_{\alpha',\omega'}=||r||_{\alpha,\omega}$.

\item[\rm (2)] Suppose that $r$ written in a basis $\mathcal{B}(i)$ is cyclically reduced and begins with a positive power of a $b$-letter. Then $r$ written in the basis $\mathcal{B}(-i-k+1)'$ is cyclically reduced and begins with a positive power of a $b'$-letter.
\end{enumerate}
\end{lemma}

\begin{proof}
Both statements can be verified straightforward by using the relations $b_i=b_{-i}'u_{-i}'$, $y_{m,i}=y_{m,-i}'^{-1}$ ($1 \leqslant m \leqslant n$).
\end{proof}

\section{Proof of Proposition \ref{prop1} for \boldmath{$\widetilde{r}$} with positive  {\boldmath $\alpha$-$\omega$}-length }

\subsection{Properties of \boldmath{$\widetilde{r}$} with positive \boldmath{$\alpha$-$\omega$}-length}

We use the following version of Magnus' Freiheitssatz:

\begin{theorem}[\textbf{Magnus' Freiheitssatz (cf. \cite{ArtMagnus})}] \label{Magnus} Let $F$ be a free group on a basis $X$, and let $g$ be a cyclically reduced word in $F$ with respect to $X$, containing a letter $x \in X$. Then the subgroup generated by $X \backslash \{x\}$ is naturally embedded into the group $F / \langle \! \langle g \rangle \! \rangle_{F}$.
\end{theorem}

By abuse of notation, we write $A / \langle \! \langle a \rangle \! \rangle$ instead of $A / \langle \! \langle a \rangle \! \rangle_{A}$, where $A$ is a group and $a \in A$.

\begin{corollary} \label{KorEinb} Let $\widetilde{r} \in N$ be a suitable element. Then we get the embeddings $G_{\alpha_{\widetilde{r}}+1,\infty} \hookrightarrow N / \langle \! \langle \widetilde{r} \rangle \! \rangle$ and $G_{-\infty, \omega_{\widetilde{r}}-1} \hookrightarrow N / \langle \! \langle \widetilde{r} \rangle \! \rangle$.
\end{corollary}

\begin{proof}
By the definition of $\alpha$-limits, $\widetilde{r}$ written in basis $\mathcal{B}^{+}(\alpha_{\widetilde{r}})$ contains at least one letter in $Y_{\alpha_{\widetilde{r}}}$, and by Definition~\ref{defsc}, the suitable element $\widetilde{r}$ is cyclically reduced in this basis. We extend $\mathcal{B}^{+}(\alpha_{\widetilde{r}})$ to a free basis of $N$ by adding the letters $\{y_{m,i} \mid 1 \leqslant m \leqslant n, i<\alpha_{\widetilde{r}}\}$. The presentation of $\widetilde{r}$ will not change by that. Now, $G_{\alpha_{\widetilde{r}}+1,\infty} \hookrightarrow N / \langle \! \langle \widetilde{r} \rangle \! \rangle$ follows immediately from Theorem~\ref{Magnus}. The other embedding follows analogously.
\end{proof}

\begin{lemma} \label{Lembw}
Suppose that $r \in N$ satisfies $||r||_{\alpha,\omega} \geqslant 1$.
Then $r$, written in the basis $\mathcal{B}^{+}(\alpha_{r})$ of $G_{\alpha_{r},\infty}$, contains at least one letter $y_{m,l}$ with $l \geqslant \omega_{r}$.
\end{lemma}

\begin{proof}
By assumption, $\alpha_r\leqslant \omega_r$.
To the contrary, suppose that $$r\in \langle b_{\alpha_r},b_{\alpha_r+1}, \dots ,b_{\alpha_r+k-1},y_{m,\alpha_r},y_{m,\alpha_r+1},\dots ,y_{m,\omega_r-1} \mid 1 \leqslant m \leqslant n\rangle.$$ Then, using relations $b_{i}=b_{i-k}u_{i-k}$, we obtain
$$r\in \langle b_{\alpha_r-k},\dots ,b_{\alpha_r-1},y_{m,\alpha_r-k},y_{m,\alpha_r-k+1}, \dots ,y_{m,\omega_r-1}\mid 1 \leqslant m \leqslant n \rangle.$$ Hence $r\in G_{-\infty, \omega_r-1}$. A contradiction.
\end{proof}

\subsection{The structure of some quotients of \boldmath{$N$}}

This section and the next one are very similar to \cite[Section 4 and 5]{ArtOneRel}, but due to some important changes we cannot skip them. Our aim in this subsection is to present $N /  \langle \! \langle \widetilde{r_{i}}, \widetilde{r}_{i+1},\dots, \widetilde{r}_{j} \rangle \! \rangle$ as an amalgamated product. We denote $w_i=b_iu_i$.

\begin{lemma} \label{LemH}
Let $\widetilde{r} \in N$ be a suitable element with $||\widetilde{r}||_{\alpha,\omega} \geqslant 1$, and let $i,j$ be two integers with $i \leqslant j$. We denote $s=\alpha_{\widetilde{r_j}}$ and $t=\omega_{\widetilde{r}_{j}}-1$. Then we have:
\begin{itemize}
\item[{\rm (1)}] $N /  \langle \! \langle \widetilde{r_{i}},\widetilde{r}_{i+1}, \dots, \widetilde{r}_{j} \rangle \! \rangle \ \cong \ G_{-\infty,t} /  \langle \! \langle \widetilde{r}_{i},\widetilde{r}_{i+1}, \dots, \widetilde{r}_{j-1} \rangle \! \rangle \underset{
    \begin{array}{ll}
    w_{t-k+1}\!\! & =b_{t+1},\vspace*{-3mm}\\
    \dots & \vspace*{-3mm}\\
    w_{t} \!\!& =b_{t+k}\\
    \end{array}
    }{\underset{G_{s,t}}{\ast}} G_{s,\infty} / \langle \! \langle \widetilde{r}_{j} \rangle \! \rangle$.
\item[{\rm (2)}] $G_{s+1,\infty}$ naturally embeds into $N/\langle \! \langle \widetilde{r}_{i},\widetilde{r}_{i+1},\dots,\widetilde{r}_{j} \rangle \! \rangle$.
\end{itemize}
\end{lemma}

Before we give a formal proof, we consider an illustrated example. This will help
to visualise a lot of technical details in the formulation of lemma.

\noindent
{\bf  Example.} Let $k=4$. We consider the element $\widetilde{r}=\widetilde{r}_0=b_4y_{2,1}y_{1,3}b_0u_0$ and the integers $i=-1$ and $j=2$.

Algorithm~\ref{algorithm} applied to $\widetilde{r}$ gives $\alpha_{\widetilde{r}_0}=1$. Its ``mirrored'' version gives $\omega_{\widetilde{r}_0}=3$.
In particular, $||\widetilde{r}||_{\alpha,\omega}=3$. Furthermore, we have $s=\alpha_{\widetilde{r}_2}=3$ and
$t=\omega_{\widetilde{r}_2}-1=4$.
Then Lemma~\ref{LemH} (1) says that
$$N /  \langle \! \langle \widetilde{r}_{-1},\widetilde{r}_0, \widetilde{r}_1 ,\widetilde{r}_2\rangle \! \rangle \ \cong \ G_{-\infty,4} /  \langle \! \langle \widetilde{r}_{-1},\widetilde{r}_0, \widetilde{r}_1  \rangle \! \rangle \underset{
    \begin{array}{rl}
    w_1\!\!\!\! & =\, b_5,\vspace*{-3mm}\\
    w_2\!\!\!\! & =\, b_6\vspace*{-3mm}\\
    w_3\!\!\!\! & =\, b_7,\vspace*{-3mm}\\
    w_4\!\!\!\! & =\, b_8
    \end{array}
    }{\underset{G_{3,4}}{\ast}} G_{3,\infty} / \langle \! \langle \widetilde{r}_2 \rangle \! \rangle.$$

In Figure 1, the word $\widetilde{r}_0=\widetilde{r}$ is pictured by the (partially dashed) line crossing the blocks $G_0,\dots,G_4$ since $\widetilde{r}$
uses letters with indices from the segment $[0,4]$.
We can represent $\widetilde{r}$ by the word $b_4y_{2,1}y_{1,3}b_4$ that uses letters with indices from the segment $[1,4]$,
and we can represent $\widetilde{r}$ by the word $b_0u_0y_{2,1}y_{1,3}b_0u_0$ that uses letters with indices from the segment $[0,3]$.
To visualise the fact that $\alpha_{\widetilde{r}_0}=1$ and $\omega_{\widetilde{r}_0}=3$, we draw a continuous line crossing the blocks $G_1,G_2,G_3$.

\begin{figure}[H]
\centering
\includegraphics[scale=0.7]{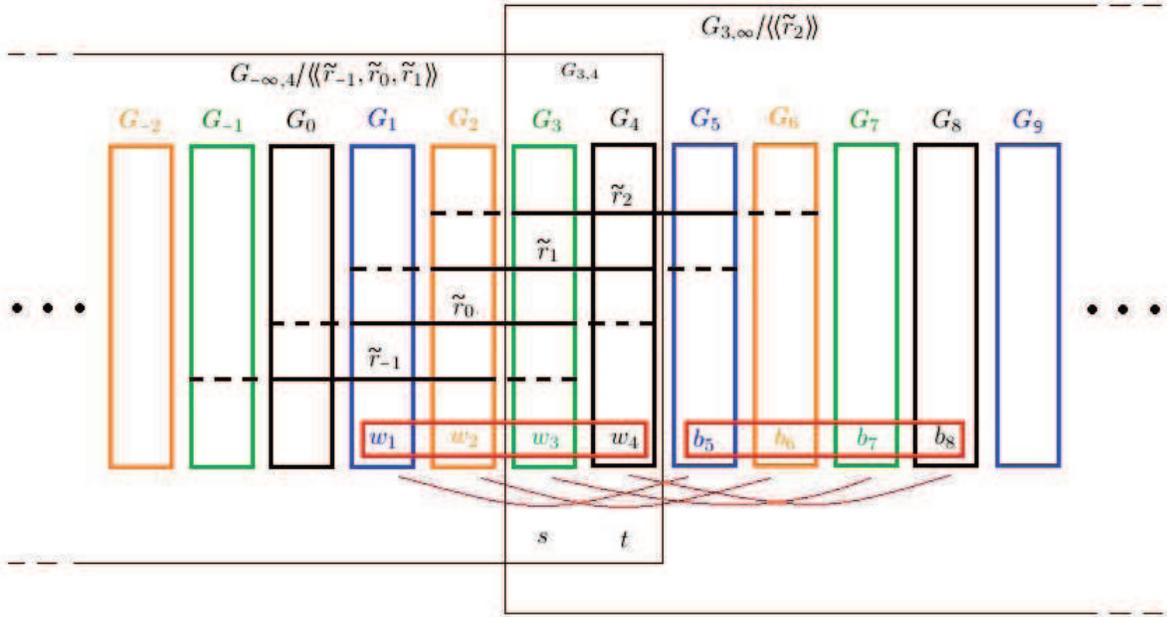}
\caption{Illustration to Lemma~\ref{LemH} for $k=4$, $\widetilde{r}=b_{4}y_{2,1}y_{1,3}b_{0}u_{0}$, $i=-1$, and $j=2$}
\label{Abb1}
\end{figure}

\begin{proof}
First, we prove that (1) implies (2). By Corollary~\ref{KorEinb}, we have $G_{s+1,\infty} \hookrightarrow G_{s,\infty} / \langle \! \langle \widetilde{r}_{j} \rangle \! \rangle$, and if (1) holds, then we have $G_{s,\infty} / \langle \! \langle \widetilde{r}_{j} \rangle \! \rangle \hookrightarrow N /  \langle \! \langle \widetilde{r}_{i},\widetilde{r}_{i+1}, \dots, \widetilde{r}_{j} \rangle \! \rangle$. The composition of these two embeddings gives (2). Now we prove (1) for fixed $i$ by induction on $j$.  \\
\textbf{Base of induction.} For $j=i$ we shall show

$$
N/  \langle \! \langle \widetilde{r}_{j} \rangle \! \rangle \ \cong \ G_{-\infty,t}
\underset{
\begin{array}{ll}
    w_{t-k+1}\!\! & =b_{t+1},\vspace*{-3mm}\\
    \dots & \vspace*{-3mm}\\
    w_{t} \!\!& =b_{t+k}\\
    \end{array}
}
{\underset{G_{s,t}}{\ast}} G_{s,\infty} / \langle \! \langle \widetilde{r}_{j} \rangle \! \rangle.\eqno{(4.1)}
$$
It suffices to show the following claim.

\emph{Claim.} Let $P$ be the subgroup of $N$ generated by $G_{s,t}\cup \{b_{t+1},\dots ,b_{t+k}\}$. Then
\vspace{-1mm}
\begin{enumerate}
\item[(a)] $P$ embeds into $G_{-\infty,t}$ and $ G_{s,\infty} / \langle \! \langle \widetilde{r}_{j} \rangle \! \rangle$.
\vspace{-1mm}
\item[(b)] The abstract amalgamated product in the right side of (4.1) is canonically isomorphic to $N/  \langle \! \langle \widetilde{r}_{j} \rangle \! \rangle$.
\end{enumerate}

\emph{Proof of the claim.} (a) Clearly, $P$ embeds into $G_{-\infty,t}$. So, we show that $P$ embeds into $G_{s,\infty} / \langle \! \langle \widetilde{r}_{j} \rangle \! \rangle$. Note that $P = \langle Y_{s} \cup \dots \cup Y_{t} \cup \{ b_{t+1},\dots ,b_{t+k} \} \rangle = \langle b_{s},b_{s+1}, \dots , b_{s+k-1}, y_{m,s},y_{m,s+1}, \dots , y_{m,t} \mid 1 \leqslant m \leqslant n \rangle$. Thus, the group $P$ is generated by the set
$$\{ b_{s},b_{s+1}, \dots , b_{s+k-1}, y_{m,s},y_{m,s+1}, \dots , y_{m,t} \mid 1 \leqslant m \leqslant n \}. \eqno{(4.2)}$$
This set is a part of the free basis $\mathcal{B}^{+}(s)$ of $G_{s,\infty}$. By Definition~\ref{defsc}, $\widetilde{r}_{j}$ written in $\mathcal{B}^{+}(s)$ is cyclically reduced. Further, by statement (1) of Lemma~\ref{Lembw}, the element $\widetilde{r}_{j}$ written in $\mathcal{B}^{+}(s)$ contains at least one letter $y_{m,\ell}$ with $1 \leqslant m \leqslant n$ and $\ell \geqslant t+1$. In particular, this letter does not lie in
the set~(4.2). By Magnus' Freiheitssatz \mbox{(Theorem~\ref{Magnus})}, $P$ embeds into $G_{s,\infty} / \langle \! \langle \widetilde{r}_{j} \rangle \! \rangle$.

Now, the amalgamated product in (4.1) is well defined. It is easy to check that the groups written in (4.1) are isomorphic by finding a common presentation. This completes the base of induction. \\
\textbf{Inductive step} $i,j \ \rightarrow i,j+1$. We need to show the formula,
where  $s=\alpha_{\widetilde{r_j}}$ and $t=\omega_{\widetilde{r}_{j}}-1$:
$$
N /  \langle \! \langle \widetilde{r}_{i},\widetilde{r}_{i+1}\dots, \widetilde{r}_{j+1} \rangle \! \rangle \ \cong \ G_{-\infty,t+1} /  \langle \! \langle \widetilde{r}_{i},\widetilde{r}_{i+1},\dots, \widetilde{r}_{j} \rangle \! \rangle \underset{
\begin{array}{ll}
    w_{t-k+2}\!\! & =b_{t+2},\vspace*{-3mm}\\
    \dots & \vspace*{-3mm}\\
    w_{t+1} \!\!& =b_{t+k+1}\\
    \end{array}
}{\underset{G_{s+1,t+1}}{\ast}} G_{s+1,\infty} / \langle \! \langle \widetilde{r}_{j+1} \rangle \! \rangle.
\eqno{(4.3)}
$$

Let $P$ be the subgroup of $N$ generated by $G_{s+1,t+1}$ and the set $\{b_{t+2},b_{t+3},\dots,b_{t+k+1} \}$. First, we prove that $P$ canonically embeds into both factors. As above $P$ embeds into $G_{s+1,\infty} / \langle \! \langle \widetilde{r}_{j+1} \rangle \! \rangle$. So we show that $P$ embeds into $G_{-\infty,t+1} /  \langle \! \langle \widetilde{r}_{i}, \widetilde{r}_{i+1}, \dots, \widetilde{r}_{j} \rangle \! \rangle$ using the following commutative diagram:
\begin{displaymath}
	\xymatrix@C=2cm@R=2cm{
        P \ \ar@{^{(}->}[r] \ \ar@{^{(}->}[d] \ar@{-->}[rrd]^{\varphi} & G_{s+1,\infty} \ \ar@{^{(}->}[r]^{\hspace{-0.9cm} by \ (2)} &  \ N /  \langle \! \langle \widetilde{r}_{i},\widetilde{r}_{i+1},\dots, \widetilde{r}_{j} \rangle \! \rangle \\
        G_{-\infty,t+1} \ \ar[rr]_{}  && \ G_{-\infty,t+1} /  \langle \! \langle \widetilde{r}_{i},\widetilde{r}_{i+1},\dots, \widetilde{r}_{j} \rangle \! \rangle \ar@{^{(}->}[u] }
\vspace{0.3cm}
\end{displaymath}
Let $\varphi$ be the composition of the canonical embedding of the subgroup $P$ into $G_{-\infty,t+1}$ and the canonical homomorphism from $G_{-\infty,t+1}$ to the factor group $G_{-\infty,t+1} /  \langle \! \langle \widetilde{r}_{i},\widetilde{r}_{i+1},\dots, \widetilde{r}_{j} \rangle \! \rangle$. It remains to prove that $\varphi$ is an embedding. Considering $P$ as a subgroup of $G_{s+1,\infty}$ and using statement (2) for $i,j$ (recall that (1) implies (2)), we have an embedding of $P$ into $N/\langle \! \langle \widetilde{r}_{i},\widetilde{r}_{i+1},\dots,\widetilde{r}_{j} \rangle \! \rangle$. Since the diagram is commutative, $\varphi: P \rightarrow G_{-\infty,t+1} /  \langle \! \langle \widetilde{r}_{i},\widetilde{r}_{i+1},\dots, \widetilde{r}_{j} \rangle \! \rangle$ is an embedding. Again, it is easy to check that the groups in (4.3) are isomorphic.
\end{proof}

Finally, we need a ``mirrored'' version of Lemma~\ref{LemH}:

\begin{lemma} \label{LemS}
Let $\widetilde{r} \in N$ be a suitable element with $||\widetilde{r}||_{\alpha,\omega} \geqslant 1$, and let $i,j$ be two integers with $i \leqslant j$. We denote $s=\alpha_{\widetilde{r}_{i}}+1$ and $t=\omega_{\widetilde{r}_{i}}$. Then we have:
\begin{itemize}
\item[{\rm (1)}] $N /  \langle \! \langle \widetilde{r_{i}},\widetilde{r}_{i+1}, \dots, \widetilde{r}_{j} \rangle \! \rangle \ \cong \ G_{-\infty,t} / \langle \! \langle \widetilde{r}_{i} \rangle \! \rangle \underset{
    \begin{array}{ll}
    w_{s-k}\!\! & =b_{s},\vspace*{-3mm}\\
    \dots & \vspace*{-3mm}\\
    w_{s-1} \!\!& =b_{s+k-1}\\
    \end{array}
    }{\underset{G_{s,t}}{\ast}} G_{s,\infty} /  \langle \! \langle \widetilde{r}_{i+1}, \widetilde{r}_{i+1},\dots, \widetilde{r}_{j} \rangle \! \rangle $.
\item[{\rm (2)}] $G_{-\infty,t-1}$ naturally embeds into $N/\langle \! \langle \widetilde{r}_{i},\widetilde{r}_{i+1},\dots,\widetilde{r}_{j} \rangle \! \rangle$.
\end{itemize}
\end{lemma}

\begin{proof} 
Due to the dual structure noticed in Remark \ref{Remdual}, the statements of Corollary \ref{KorEinb}, Lemma \ref{Lembw} and Lemma \ref{LemH} also hold for the dual objects. By rewriting the dual version of Lemma \ref{LemH} with the help of Lemma \ref{LemDual}, we get the desired statement.
\end{proof}

\subsection{Conclusion of the proof for \boldmath{$\widetilde{r}$} with positive \boldmath{$\alpha$-$\omega$}-length} \label{conclusion1}

By assumption of Proposition~\ref{prop1}, we have two elements $r,s \in H$ with the same normal closure and $r_{x}=0$. Thus, $r,s$ are elements of $N$. By Lemma~\ref{lemdef}, we can choose conjugates $\widetilde{r},\widetilde{s}$ of $r,s$ such that $\widetilde{r}$ and $\widetilde{s}$ are suitable elements. Normal closures are invariant under conjugation. So, without loss off generality, we can replace $r,s$ by $\widetilde{r},\widetilde{s}$. In this section we assume $||\widetilde{r}||_{\alpha,\omega} \geqslant 1$. Since the normal closures of $\widetilde{r}$ and $\widetilde{s}$ in $H$ are equal, the normal closures of $\mathcal{R}:=\{ \widetilde{r}_{i} \ | \ i \in \mathbb{Z} \}$ and $\mathcal{S}:=\{ \widetilde{s}_{i} \ | \ i \in \mathbb{Z} \}$ in $N$ are equal. In particular, $\widetilde{s}_{0}$ is trivial in $N / \langle \! \langle \mathcal{R} \rangle \! \rangle$. Thus, there are indices $i,j \in \mathbb{Z}$ such that $\widetilde{s}_{0}$ is trivial in $N / \langle \! \langle \widetilde{r}_{i},\widetilde{r}_{i+1},\dots,\widetilde{r}_{j} \rangle \! \rangle$. We choose a pair $i,j$ with this property and $j-i$ minimal. By Lemma~\ref{LemH}~(1), we have
$$
N /  \langle \! \langle \widetilde{r}_{i},\widetilde{r}_{i+1},\dots, \widetilde{r}_{j} \rangle \! \rangle \ \cong \ G_{-\infty,\omega_{\widetilde{r}_{j}}-1} /  \langle \! \langle \widetilde{r}_{i},\widetilde{r}_{i+1},\dots, \widetilde{r}_{j-1} \rangle \! \rangle \ \underset{A}{\ast} \ G_{\alpha_{\widetilde{r}_{j}},\infty} / \langle \! \langle \widetilde{r}_{j} \rangle \! \rangle\eqno{(4.4)}
$$
for some subgroup $A$.
\begin{lemma}\label{new} Suppose that $||\widetilde{r}||_{\alpha,\omega} \geqslant 1$. Then
\vspace*{-2mm}
\begin{enumerate}
\item[{\rm (1)}] $\omega_{\widetilde{s}_{0}} \geqslant \omega_{\widetilde{r}_{j}}\geqslant \omega_{\widetilde{r}_{i}}$,
\vspace*{-1mm}
\item[{\rm (2)}] $\alpha_{\widetilde{s}_0}\leqslant \alpha_{\widetilde{r}_i}\leqslant \alpha_{\widetilde{r}_j}$,
\vspace*{-1mm}
\item[{\rm (3)}] $
||\widetilde{s}||_{\alpha,\omega} \geqslant ||\widetilde{r}||_{\alpha,\omega}.
$
\end{enumerate}
\end{lemma}

\begin{proof} (1) We prove $\omega_{\widetilde{s}_{0}} \geqslant \omega_{\widetilde{r}_{j}}$. Suppose the contrary. Then $\widetilde{s}_{0} \in G_{-\infty,\omega_{\widetilde{r}_{j}}-1}$. Since $\widetilde{s}_{0}$ is trivial in $N /  \langle \! \langle \widetilde{r}_{i},\widetilde{r}_{i+1},\dots, \widetilde{r}_{j} \rangle \! \rangle$, it is also trivial in the left factor of the amalgamated product (4.4). Therefore, $\widetilde{s}_0$ is trivial in $N/  \langle \! \langle \widetilde{r}_{i},\widetilde{r}_{i+1}, \dots, \widetilde{r}_{j-1} \rangle \! \rangle$. This contradicts the minimality of $j-i$. The inequality \mbox{$\omega_{\widetilde{r}_{j}}\geqslant \omega_{\widetilde{r}_{i}}$} follows by Lemma \ref{LemEig} since $j \geqslant i$. Inequalities (2) can be proved analogously with the help of Lemma~\ref{LemS}~(1). Inequality (3) follows straightforward from (1) and (2).
\end{proof}

Since $||\widetilde{s}||_{\alpha,\omega} \geqslant ||\widetilde{r}||_{\alpha,\omega} \geqslant 1$, we have $||\widetilde{r}||_{\alpha,\omega} \geqslant ||\widetilde{s}||_{\alpha,\omega}$ by symmetry and therefore $||\widetilde{s}||_{\alpha,\omega} = ||\widetilde{r}||_{\alpha,\omega}$. By Lemma~\ref{LemEig}, this is equivalent to $||\widetilde{s}_{0}||_{\alpha,\omega} = ||\widetilde{r}_{i}||_{\alpha,\omega}$.
This and the first two statements of Lemma~\ref{new} imply that
$\alpha_{\widetilde{s}_{0}} = \alpha_{\widetilde{r}_{i}} = \alpha_{\widetilde{r}_{j}}$, and $\omega_{\widetilde{r}_{j}}=\omega_{\widetilde{r}_{i}} = \omega_{\widetilde{s}_{0}}$, in particular, $j=i$. Therefore, $\widetilde{s}_{0}$ is trivial in $N / \langle \! \langle \widetilde{r}_{i} \rangle \! \rangle$, and the index $i$ is determined by $\alpha_{\widetilde{s}_{0}} = \alpha_{\widetilde{r}_{i}}$. One can prove in the same way that $\widetilde{r}_{i}$ is trivial in $N / \langle \! \langle \widetilde{s}_{0} \rangle \! \rangle$. Thus, $\langle \! \langle \widetilde{s}_{0} \rangle \! \rangle_{N} =  \langle \! \langle \widetilde{r}_{i} \rangle \! \rangle_{N}$. From Section~\ref{leftright} we know that $N$ is a free group. So by \mbox{Theorem~\ref{Magnus}}, $\widetilde{s}_{0}$ is conjugate to $\widetilde{r}_{i}^{\pm1}$ in $N$. Finally, $\widetilde{r}$ is conjugate to $\widetilde{s}^{\pm1}$ in $H$, and the proof of Proposition~\ref{prop1} in the case $||\widetilde{r}||_{\alpha,\omega} \geqslant 1$ is completed. \hfill $\qed$

\section{Proof of Proposition \ref{prop1} for \boldmath{$\widetilde{r}$} with non-positive  {\boldmath $\alpha$-$\omega$}-length}

\subsection{Properties of \boldmath{$\widetilde{r}$} with non-positive \boldmath{$\alpha$-$\omega$}-length}

\begin{lemma} \label{LemEig2}
Let $r$ be an element in $N$ with $||r||_{\alpha,\omega} < 1$.
Then the reduced word representing $r$ in basis $\mathcal{B}^{+}(\alpha_{r})$ contains only $b_{i}$'s.
\end{lemma}

\begin{proof} We have $\omega_r<\alpha_r$. By definition,  $r\in G_{-\infty,\omega_r}$ and $r\in G_{\alpha_r,\infty}$.\\
$\bullet$ Let $r^{-}$ be the reduced word representing $r$ in the basis $\mathcal{B}^{-}(\omega_r)$ of $G_{-\infty,\omega_r}$. Recall that
$$\mathcal{B}^{-}(\omega_r)=\{b_i\,|\, i\in I \}  \cup \{ y_{m,\ell} \mid 1 \leqslant m \leqslant n, \ell \leqslant \omega_r \},$$
where $I=\{\omega_r-k+1,\dots ,\omega_r\}$.\\
\noindent
$\bullet$ Let $r^{+}$ be the reduced word representing $r$ in the basis $\mathcal{B}^{+}(\alpha_r)$ of $G_{\alpha_r,\infty}$. Recall that
$$\mathcal{B}^{+}(\alpha_r)=\{b_j\,|\, j\in J\}  \cup \{y_{m,\ell} \mid 1 \leqslant m \leqslant n, \ell \geqslant \alpha_r \},$$
where $J=\{\alpha_r,\dots,\alpha_r+k-1\}$.

For each $i\in I$, there is a unique $j\in J$ such that $j\equiv i \mod k$.
Then $r^{+}$ can be obtained from $r^{-}$ by replacements of all letters $b_i$ by
$b_ju_{j-k}^{-1}u_{j-2k}^{-1}\dots u_i^{-1}$ followed by reduction (see Remark~\ref{remchange}). The second indices of $y$-letters appearing in these
replacements lie in
the interval $(-\infty, \max J -k]= (-\infty, \alpha_r-1]$.
The second indices of $y$-letters of $r^{-}$ lie in the interval $(-\infty, \omega_r]\subseteq (-\infty, \alpha_r-1]$.
Therefore, the second indices of $y$-letters of $r^{+}$ (if exist) lie in $(-\infty, \alpha_r-1]$.
Hence, $r^{+}$ does not have $y$-letters.
\end{proof}

\subsection{Conclusion of the proof for \boldmath{$\widetilde{r}$} with non-positive \boldmath{$\alpha$-$\omega$}-length} \label{conclusion2}

As in Section \ref{conclusion1}, we can replace the elements $r,s \in N$ from Proposition~\ref{prop1} by suitable elements~$\widetilde{r}, \widetilde{s}$.
In this section, we consider the case $||\widetilde{r}||_{\alpha,\omega} < 1$. Assuming $||\widetilde{s}||_{\alpha,\omega} \geqslant 1$, we immediately get a contradiction by applying Lemma~\ref{new} with reversed roles of $\widetilde{r}$ and $\widetilde{s}$. Thus, we have $||\widetilde{s}||_{\alpha,\omega} < 1$.

By Lemma~\ref{LemEig2}, there are presentations of $\widetilde{r}$ and $\widetilde{s}$ in $N$ which use only $b_i$, $i\in \mathbb{Z}$. So by means of the relations $b_{i}=x^{-i}bx^{i}$, we get presentations of $\widetilde{r}$ and $\widetilde{s}$ as elements in $$H=\langle x,b,y_{1}, \dots , y_{n} \ | \ [x^{k},b]u \rangle$$ that only use $x$ and $b$.

Thus, $\widetilde{r},\widetilde{s}\in \widetilde{H}$, where $\widetilde{H}$ is the subgroup of $H$ generated by $x$ and $b$. We have the presentations

$$
\widetilde{H}=\langle x,b \rangle \cong F_{2} \ \ \text{and} \ \ H=\widetilde{H} \underset{[x^{k},b]=u^{-1}}{\ast} \langle y_{m} \mid 1 \leqslant m \leqslant n \rangle, \eqno{(5.1)}
$$

where $F_{2}$ is the free group in two generators and the isomorphism to $F_{2}$ follows from the Magnus' \mbox{Freiheitssatz}.

{\bf Convention.} For $h\in \widetilde{H}$, let $(\!(h)\!)$ be the normal closure of $h$ in $\widetilde{H}$ and
let $\langle\! \langle h\rangle\! \rangle$ be the normal closure of $h$ in $H$.

Consider the canonical homomorphism $\iota: \widetilde{H} \rightarrow \widetilde{H} / (\! ( \widetilde{r})\! )$.
There are two cases:

{\bf Case 1.} $\ [x^{k},b]^{p} \notin (\! ( \widetilde{r} )\! )$
for all $p \geqslant 1$. \\
In this case $\iota \mid_{\langle [x^{k},b] \rangle}: \widetilde{H} \rightarrow \widetilde{H} /
(\! (\widetilde{r} ) \!)$ is an embedding. We identify the subgroup $\langle [x^{k},b] \rangle$ of $\widetilde{H}$ with its image in $\widetilde{H}/(\! ( \widetilde{r} ) \! )$. Then
\begin{eqnarray*}
H / \langle \! \langle \widetilde{r} \rangle \! \rangle &=& \langle x,b,y_{1},\dots,y_{n} \ | \ [x^{k},b]u, \ \widetilde{r} \rangle \ \ = \ \ \widetilde{H} / (\!( \widetilde{r} )\!) \underset{[x^{k},b]=u^{-1}}{\ast} \langle y_{m} \mid 1 \leqslant m \leqslant n \rangle \ \ = \ \ \widetilde{H} / ( \!( \widetilde{r} )\!) \underset{\mathbb{Z}}{\ast} \mathbb{F}_{n}.
\end{eqnarray*}

{\bf Case 2.} $[x^{k},b]^{p} \in ( \! ( \widetilde{r} ) \! )$ for some $p\geqslant 1$. \\
We choose the smallest $p$ with this property.
Using Tietze transformations, we write:
\begin{eqnarray*}
H / \langle \! \langle \widetilde{r} \rangle \! \rangle &=& \langle x,b,y_{1},\dots,y_{n} \ | \ [x^{k},b]u, \ \widetilde{r} \, \rangle \ = \ \langle x,b,y_{1},\dots , y_{n} \ | \ [x^{k},b]u, \ \widetilde{r}, \ [x^{k},b]^{-p} \, \rangle \\
&=& \langle x,b,y_{1}, \dots , y_{n} \ | \ [x^{k},b]u, \ \widetilde{r}, \ u^{p} \, \rangle \ = \ \widetilde{H} / (\! ( \widetilde{r} )\! ) \underset{[x^{k},b]=u^{-1}}{\ast} \langle y_{1}, \dots , y_{n} \ | \ u^{p} \, \rangle.
\end{eqnarray*}
Thus, in both cases, we have the embedding $\widetilde{H} / ( \! (\widetilde{r} )\! ) \hookrightarrow H / \langle \! \langle \widetilde{r} \rangle \! \rangle$. \\

We already know that $\widetilde{s}$ is an element of the subgroup $\widetilde{H}$ of $H$. Further, $\widetilde{s}$ is trivial in $H / \langle \! \langle \widetilde{r} \rangle \! \rangle$ since $\langle \! \langle \widetilde{r} \rangle \! \rangle = \langle \! \langle \widetilde{s} \rangle \! \rangle$. Now, we use the embedding $\widetilde{H} / (\! ( \widetilde{r} )\! ) \hookrightarrow H / \langle \! \langle \widetilde{r} \rangle \! \rangle$ to conclude that $\widetilde{s}$ is trivial in $\widetilde{H} / ( \! ( \widetilde{r} ) \! )$. By symmetry, $\widetilde{r}$ is trivial in $\widetilde{H} / ( \! ( \widetilde{s} )\! )$, and we get $(\! ( \widetilde{r} ) \! ) = (\! ( \widetilde{s} )\! )$.

As we have seen in (5.1), the group $\widetilde{H}$ is isomorphic to the free group of rank 2, and this group possesses the Magnus property (cf. \cite{ArtMagnus}). Hence, $\widetilde{r}$ is conjugate to $\widetilde{s}^{\pm1}$ in $\widetilde{H}$, and since $\widetilde{H} \hookrightarrow H$, the element $\widetilde{r}$ is conjugate to $\widetilde{s}^{\pm1}$ in $H$. \hfill $\qed$ \\

\textbf{Acknowledgements.} This work is based on my M.Sc. thesis, Heinrich Heine Universit\"at D\"usseldorf, 2015. I would like to express my special thanks to my supervisor Professor Oleg Bogopolski for the subject proposal and the support during and after the preparation of my M.Sc. thesis. I also want to thank Professor Benjamin Klopsch for his editorial help.

\bibliography{Literatur}
    \bibliographystyle{alpha}

\end{document}